\documentclass[11pt,twoside]{amsart}

\usepackage{color}
\usepackage{comment}

\usepackage{mathrsfs,amsfonts,amsmath,amssymb,amsthm}
\usepackage{dsfont}

\usepackage{comment}
\usepackage{enumitem}

\usepackage{url}
\usepackage[breaklinks]{hyperref}
\hypersetup{
colorlinks = true,
 linkcolor = blue,
}

\usepackage{todonotes}

\usepackage{enumitem} 

\newtheorem{proposition}{Proposition}
\newtheorem{theorem}{Theorem}
\newtheorem{lemma}{Lemma}
\newtheorem{corollary}{Corollary}
\newtheorem{remark}{Remark}
\newtheorem{exm}{Example}
\newtheorem{conjecture}{Conjecture}

\theoremstyle{definition}
\newtheorem{definition}{Definition}

\newtheoremstyle{named}{}{}{\itshape}{}{\bfseries}{.}{.5em}{\thmnote{#3 }#1}
\theoremstyle{named}
\newtheorem*{namedlemma}{Lemma}

\newtheoremstyle{author}{}{}{\itshape}{}{\bfseries}{.}{.5em}{\thmnote{#1 }(#3)}
\theoremstyle{author}

\newtheorem*{authortheorem}{Theorem}

\newtheoremstyle{custom}{}{}{\itshape}{}{\bfseries}{}{.5em}{\thmnote{#3}.}
\theoremstyle{custom}

\newcommand{\wiener}[1]{\left\lVert#1\right\rVert_{\omega}}

\newcommand{\oversetgap}[3][0.6ex]{%
  \overset{\raisebox{#1}{$\scriptstyle #2$}}{#3}%
}

\def\N{\mathbb{N}}
\def\Z{\mathbb{Z}}
\def\R{\mathbb{R}}
\def\C{\mathbb{C}}
\def\F{\mathbb{F}}

\def\E{\mathsf{E}}

\def\eps{\varepsilon}

\def\d{\delta}

\def\Gr{{\mathbf G}}

\def\le{\leqslant}
\def\ge{\geqslant}

\def\ov{\overline}

\def\FF{\widehat}

\def\bp{\bigskip}

\def\({\big (}
\def\){\big )}

\author{Semchankau A.S. and Shkredov I.D.}
\title{Structure theory of set addition with two operations
}
\date{}

\begin{document}
	\maketitle

\begin{center}
	Annotation.
\end{center}

{\it \small
    We take the first step toward a structure theory that includes both operations of a ring $\mathcal{R}$. More precisely, we prove a series of inverse results for the structure of sets $A\subseteq \F_p$ such that, under certain conditions on integers $r_1, \dots, r_k$, one has  
    $|A^{r_1} + \dots + A^{r_k}| \ll \sqrt[k]{p^{k-1} |A|}$. 
}
\\
    
\tableofcontents

\section{Introduction}

Given an abelian group $\Gr$ and two sets $A, B\subseteq \Gr$, define 
the {\it sumset} of $A$ and $B$ as 
\begin{equation}\label{def:A+B_intr}
 A+B:=\{a+b ~:~ a\in{A}, b\in{B}\}.
\end{equation}
In a similar way we define the \emph{difference sets} and \emph{higher sumsets}, e.g., $A-A$, $A+A-A$, and so on.
The \emph{doubling constant} of a finite set $A$ is 
\begin{equation}\label{def:doubling}
 \mathcal{D} [A] := \frac{|A+A|}{|A|}, 
\end{equation}
an important additive--combinatorial characteristic of $A$. 

The study of the structure of sumsets is a fundamental problem in classical additive combinatorics \cite{TV}, and this relatively new branch of mathematics provides us with a number of excellent results.
For example, Freiman's theory on the structure of the family of sets having small doubling --- i.e., sets for which the ratio \eqref{def:doubling} is ``bounded'' $\mathcal{D} [A] \le K$, where $K$ is a fixed or slowly growing parameter --- seems especially complete (see, e.g., \cite{Freiman_book}, \cite[Section 5]{TV} and recent results in \cite{GGMT_Marton_bounded}).

Further, given a ring $\mathcal{R}$ and two sets $A, B \subseteq \mathcal{R}$, one can similarly define the \emph{product set} $AB$ and higher products/sumsets, involving both operations of the ring such as $AA+A$ or $A(A+A)$.
The interaction of these additive and multiplicative structures on the ring $\mathcal{R}$ is the subject of arithmetic combinatorics, or, in other words, the \emph{sum--product phenomenon}.
This area typically studies the behavior of 
\[
\max\{ |A+A|, |AA|\}
\]
under various assumptions on $A$, see \cite{TV}.
More generally, given two polynomials $f,g \in \mathcal{R} [x_1,\dots,x_n]$ 
one can consider the sets 
\begin{equation}\label{def:rational_sp}
    \left\{ \frac{f(a_1,\dots,a_n)}{g(a_1,\dots,a_n)} ~:~ a_1,\dots, a_n \in A \right\}.
\end{equation}
In this work we take a first step towards a structure theory that simultaneously involves \emph{both} operations $+$ and $\times$, which is the novelty of the paper.

\subsection{Main conjecture}
In this paper we limit ourselves to considering the field $\F_p$.
Let $p$ be a prime number and let $A \subseteq \F_p$. 
Denote by
\[
 A^* := A^{-1} = \{1/a \bmod p : a \in A, a \neq 0\}
\]
its inverse set. In~\cite{Semchankau_wrappers}, we proved the lower bound
\begin{equation}\label{eq:lower-bound-a-a*}
 |A + A^*| \ge (1 - o(1)) \min\{ 2\sqrt{p|A|}, p \},
\end{equation}
valid whenever $|A| \gg p(\log p)^{-1/2 + o(1)}$. 
The archetypal 
question of the present work is the corresponding \emph{inverse problem} for a simple rational function \eqref{def:rational_sp}:
given that
\begin{equation}\label{as:A+A*_intr}
 |A + A^*| \le K\sqrt{p|A|},
\end{equation}
for a fixed parameter $K\ge 2$, what structural information can we deduce about $A$? 

To construct examples of sets $A$ such that the cardinality of $A + A^*$ obeys \eqref{as:A+A*_intr}
consider well--structured sets 
$P, Q \subseteq \F_p$ such that $P$ and $Q^*$ behave ``independently'':
\[
 |P + Q| \le K\sqrt{|P||Q|}, \qquad |P \cap Q^*| \approx \frac{1}{p}|P||Q|.
\]
This situation arises when, for instance, $P$ and $Q$ are sufficiently large arithmetic progressions: 
\[
 |P \cap Q^*| = \frac{1}{p}|P||Q| + O\left(\sqrt{p}(\log p)^2\right) = \frac{1 + o(1)}{p}|P||Q|, \qquad \text{provided that } |P|, |Q| \gg p^{3/4}\log p.
\]

Now set $A := P \cap Q^*$. Then $A + A^* \subseteq P + Q$, and hence
\[
 |A + A^*| \le K\sqrt{|P||Q|} \approx K\sqrt{p|A|}.
\]

Finally, if we add a constant number $C$ of arbitrary elements to $A$ to form $A_1$, then
\[
 |A_1 + A_1^*| 
 \le |A + A^*| + 2C|A| + C^2
 \sim (K + 2C\sqrt{\alpha} + o(1))\sqrt{p|A|},
\]
hence the modified constant $K_1 := |A_1 + A_1^*| / \sqrt{p|A_1|}$ remains close to $K$ provided $C \ll K/\sqrt{\alpha}$, where $\alpha = |A| / p$.


This heuristic suggests the following: 
\begin{conjecture}
Let $A \subseteq \F_p$ be a set of size at least $p^{1 - c}$ for some absolute constant $c > 0$, and suppose that
\[
 |A + A^*| \le K\sqrt{p|A|}.
\]
Then there exist sets $P, Q \subseteq \F_p$ such that, after removing $O_{K,\alpha}(1)$ elements from $A$, the following holds:
\begin{itemize}[itemsep=0.5em]
 \item $A \subseteq P \cap Q^*$;
 \item $|P \cap Q^*| = \tfrac{1 + o(1)}{p}|P||Q|$;
 \item $|P + Q| \le (1 + o(1)) K\sqrt{p|A|}$.
\end{itemize}
\end{conjecture}

The first two conditions imply
\[
p|A| \le (1 + o(1))|P||Q|.
\]
Combined with the third condition, this gives
\[
|P + Q| \le (1 + o(1))K\sqrt{|P||Q|}, 
\]
and sets $P$ and $Q$ can be further characterized using Freiman’s Theorem~\cite{Freiman_book}.

\subsection{Main results}

We verify several variants of this conjecture for dense sets $A$. 
The methods we develop for this conjecture turn out to be applicable in a more general setting, where the exponent configuration $(1, -1)$ in \eqref{as:A+A*_intr} is replaced by a tuple $(r_1, \ldots, r_k)$ satisfying certain arithmetic conditions. Some definitions are therefore to be introduced. 

Let $r$ be an exponent (non--zero integer), and $A \subseteq \F_p$. 
Define $A^r := \{a^r : a \in A\}$ (note that this notation is not standard, and should not be confused with the $r$-fold product set or the set of $r$-tuples). Since we work in $\F_p$, an exponent $r$ might be treated as an element of $\Z_{p - 1}$. 
\begin{definition}
Exponent $r$ is called \emph{generic} if $(r, p - 1) = O(1)$, and \emph{coprime} if $(r, p - 1) = 1$. More generally, a tuple $(r_1, \ldots, r_k)$ of distinct exponents satisfying $(r_i - r_j, p - 1) < p^{1 - \delta}$ for some $\delta > 0$ is \emph{generic} or \emph{coprime} if all individual $r_i$'s are generic or coprime, respectively.
\end{definition}

Define $\sqrt[r]{A} := \{x: x^r \in A\}$. We will work only with generic $r$, implying that always $|\sqrt[r]{A}| \leqslant (r, p - 1)|A| \ll |A|$.

The following result generalizes the lower bound~\eqref{eq:lower-bound-a-a*}.

\begin{theorem}[Lower Bound]\label{thm:lower-bound}
Let $A \subseteq \F_p$ be a set with $|A| \gg p$, and let $(r_1, \ldots, r_k)$ be a generic tuple. Then
\[
 \big|A^{r_1} + \cdots + A^{r_k}\big|
 \ge (1 - o(1)) \min \big(k p^{\frac{k-1}{k}}|A|^{1/k}, p\big).
\]
\end{theorem}

The following theorem confirms the conjecture for almost all, or `99\%' elements of $A$, that is, after removing only $o(p)$ of them.

\begin{theorem}[$99\%$ Characterization of $A$]\label{thm:99-percent}
Let $A \subseteq \F_p$ be a set with $|A| \gg p$, and let $(r_1, \ldots, r_k)$ be a generic tuple. If
\[
 |A^{r_1} + \cdots + A^{r_k}| 
 \le K p^{\frac{k-1}{k}}|A|^{1/k},
\]
then there exist sets $P_1, \ldots, P_k \subseteq \F_p$ such that, 
after removing $o(p)$ elements from~$A$, the following holds:
\begin{itemize}[itemsep=0.5em]
 \item $A \subseteq \sqrt[r_1]{P_1} \cap \cdots \cap \sqrt[r_k]{P_k}$;
 \item $\big|\sqrt[r_1]{P_1} \cap \cdots \cap \sqrt[r_k]{P_k}\big|
  = \frac{1 + o(1)}{p^{k-1}}|P_1|\cdots|P_k|$;
 \item $|P_1 + \cdots + P_k|
  \le (1 + o(1)) K p^{\frac{k-1}{k}}|A|^{1/k}$.
\end{itemize}
\end{theorem}

In fact, the sets $P_1, \ldots, P_k$ obtained in Theorem~\ref{thm:99-percent} (and in Theorem~\ref{thm:100-percent} below) are quite specific and possess a good additive structure. More precisely, they are well approximated by sets of small Wiener norm; see the proofs and Section~\ref{sec:comb-wiener}.

\bp 

We can also obtain a structural characterization for the \emph{entire} set~$A$,
after removing only $O(p/|A|)$ elements, although with a suboptimal dependence of $f(K)$ on $K$, see Theorem \ref{thm:100-percent} below. We also have to make additional assumptions on the exponents.

\begin{definition}
A generic tuple $(r_1, \ldots, r_k)$ is called \emph{bounded} 
if there exists a scaling parameter $L = O(1)$, coprime to $p - 1$, 
such that all residues $r_i L \bmod (p - 1)$ have order $O(1)$. 
In other words, there is an absolute constant $C > 0$ for which
\[
 |r_i L \bmod (p - 1)| \le C,
 \qquad i = 1, \ldots, k,
\]
where residues modulo $p - 1$ are identified with the interval 
$[-\frac{p-1}{2}, \frac{p-1}{2}]$.
\end{definition}

\begin{theorem}[$100\%$ Characterization of $A$]\label{thm:100-percent}
Let $A \subseteq \F_p$ be a set of size $|A| = \alpha p$ with $\alpha \gg 1$, 
and let tuple $(r_1, \ldots, r_k)$ be coprime and bounded such that
\[
 |A^{r_1} + \cdots + A^{r_k}|
 \le K p^{\frac{k-1}{k}}|A|^{1/k}.
\]
Then there exist sets $P_1, \ldots, P_k \subseteq \F_p$ such that, 
after removing $O(1/\alpha)$ elements from~$A$, the following holds:
\begin{itemize}[itemsep=0.5em]
 \item $A \subseteq \sqrt[r_1]{P_1} \cap \cdots \cap \sqrt[r_k]{P_k}$;
 \item $\big|\sqrt[r_1]{P_1} \cap \cdots \cap \sqrt[r_k]{P_k}\big|
  = \frac{1 + o(1)}{p^{k-1}}|P_1|\cdots|P_k|$;
 \item $|P_1 + \cdots + P_k| \le f(K) p^{\frac{k-1}{k}}|A|^{1/k}$,
\end{itemize}
where $f(K) \ll K^{O(1)}$.
\end{theorem}

Although both the density $\alpha$ and the parameter $K$ are assumed to be `constants', they have different scales --- one may think the case $K = 10$ and $\alpha = 10^{-100}$.

The main ingredient in the proof of Theorem~\ref{thm:lower-bound} is the observation that algebraic transformations $\sqrt[r]{W}$ of sets with small Wiener norm intersect in a manner similar to \emph{independent} sets. 
The main ingredient in Theorem~\ref{thm:99-percent} is Green's Arithmetic Regularity Lemma. 
Finally, the main ingredient in the proof of Theorem~\ref{thm:100-percent} is that, under additional assumptions, the intersections
\[
W \cap \sqrt[r_1]{W_1} \cap \ldots \cap \sqrt[r_k]{W_k}
\]
behave similarly to \emph{pseudorandom} subsets of $W$.

\newpage 


\noindent
\textbf{Organization of the paper}

\begin{itemize}
 \item \textbf{Section~\ref{sec:prelims}} recalls the basic properties of the Fourier transform, the Wiener norm, exponential sums, sumsets, and proves various auxiliary propositions.

 \item \textbf{Section~\ref{sec:comb-wiener}} develops combinatorial tools based on the Wiener norm. We introduce the \emph{algebraic intersection property} and establish it for sets, approximated by sets of small Wiener norm: their algebraic transformations intersect as independent sets.
 We also recall the concept of \emph{wrappers} from~\cite{Semchankau_wrappers}, 
 which serve as a source of such sets and help to extend the algebraic intersection property to sets of popular differences.

 \item \textbf{Sections~\ref{sec:lower-bound}} and~\textbf{\ref{sec:99-percent}} contain the proofs of Theorems~\ref{thm:lower-bound} and~\ref{thm:99-percent}, respectively. These proofs are relatively short and use only a subset of tools introduced earlier.

 \item \textbf{Section~\ref{sec:new_norm}} defines an additive distance $\rho(f,g)$ between functions and shows that intersection of algebraic transformations of sets with algebraic intersection property not only has `expected' cardinality, but also exhibits strong pseudorandom properties.

 \item \textbf{Section~\ref{sec:sumsets}} studies sumsets of the form $Y + T$, where $T$ is a dense subset of intersections such as $P \cap Q^*$. 
 We first show that $|Y + T| \gg |P|$, and then derive a Ruzsa-type structural description of sets $Y$ for which $|Y + T| \ll |P|$. The first result heavily relies on the previous section. 

 \item \textbf{Section~\ref{sec:100-percent}} concludes with the proof of the 
 \emph{$100\%$ Characterization Theorem}~\ref{thm:100-percent}. 
 This argument builds on the $99\%$ case and relies on the results of the preceding section.

 \item \textbf{Appendix~\ref{sec:appendix_dist}} includes some further facts on the 
 additive distance introduced in Subsection \ref{ssec:new_normI} and other metrics. 
\end{itemize}

\noindent
\textbf{Acknowledgements}. The first author would like to thank Boris Bukh and Prasad Tetali for useful discussions. The first author was supported in part by Prasad Tetali's NSF grant DMS-2151283.

\section{Definitions and Preliminary Results}
\label{sec:prelims}

\subsection{Notation}
It is always assumed, unless stated otherwise, that all sets are subsets of the abelian group $\Gr = \F_p$, the field of residues modulo a (very) large prime $p$, with the usual addition and multiplication.
We write $\F^*_p$ for $\F_p \setminus \{ 0 \}$ and $\ov{\F}_p$ for the algebraic closure of $\F_p$. 

Throughout the work, $o(1)$ means hidden dependency on $p$. Both $o(1)$ and $O(1)$ may also have hidden dependence on the parameter $k$ (the number of sets in the sumset under consideration) and, depending on the context, other constants. To highlight the dependency, those constants can be added as indices, e.g. $O_{X}(1)$. 
By definition, the terms $o(\ldots)$ and $O(\ldots)$ may be negative or positive. In case we want to highlight that the quantity of interest is actually \emph{negative}, we may write $-o(\ldots)$ or $-O(\ldots)$. 

The signs $\ll$ and $\gg$ are the usual Vinogradov symbols. In case we want to highlight the dependence of underlying constant $c$ on a parameter $X$, we may write $\ll_X$ and $\gg_X$. We write $f \asymp g$ if $f \gg g \gg f$.
There might be several scenarios for using these signs:
\begin{itemize}
    \item ``If $|A| \gg p$, then $\ldots$'' might be restated as ``for any constant $c > 0$, if $|A| \geqslant cp$, and $p$ is sufficiently large, meaning $p > p(c)$, then $\ldots$''
    \item ``If condition $\mathbf{X}$ holds, then $A \gg B$'' might be restated as ``there exists an absolute constant $c>0$ (depending on a situation, either sufficiently small or sufficiently large), such that condition $\mathbf{X}$ implies $A \geqslant cB$''.
\end{itemize}
The particular scenario should always be clear from the context.

All logarithms are to base $e$. Denote $e(x) = e^{2\pi ix/p}$. 
The complement of the set $A$ is denoted by $A^c$.
The symbol $A \sqcup B$ denotes the disjoint union of sets $A$ and $B$.

\subsection{Fourier Transform}

Let $\Gr$ be a finite abelian group and $\FF{\Gr}$ its dual group. By definition, $\FF{\Gr}$ consists of \emph{additive characters}, i.e. functions $\xi : \Gr \to \C$ such that $\xi(a + b) = \xi(a)\xi(b)$, $\forall a, b \in \Gr$.
In the context of this work, $\Gr$ is always $\F_p$ and, conveniently, $\FF{\F}_p \cong \F_p$. All additive characters take the form $\xi_a(x) = e(ax)$ for an arbitrary $a \in \F_p$ (with $a=0$ corresponding to the \emph{trivial} additive character $\xi_0(x) \equiv 1$), and therefore one might identify $\FF{\F}_p$ with $\F_p$ by $a \leftrightarrow \xi_a(x)$.

Having any function $f:\F_p \to \mathbb{C}$ and $\xi \in \FF{\F}_p$ define the \emph{Fourier transform} of $f$ at $\xi$ by the formula 
\begin{equation}\label{f:Fourier_representations}
\FF{f} (\xi) = \sum_{x\in \F_p} f(x) \ov{\xi (x)}.
\end{equation}
Because of $\F_p \cong \FF{\F}_p$ we can, with abuse of notation, write $\FF{f}(a) := \FF{f}(\xi_a)$.
The \emph{Parseval identity} is 
\begin{equation}\label{F_Par}
p\sum_{x\in \F_p} |f(x)|^2 = 
\sum_{\xi \in \FF{\F}_p} \big|\widehat{f} (\xi)\big|^2,
\end{equation}
and the identity
\begin{equation}\label{f:inverse}
f(x) = \frac{1}{p} \sum_{\xi \in \FF{\F}_p} \FF{f} (\xi) \xi(x).
\end{equation}
is called the \emph{inverse} formula. 
Given two functions $f,g : \F_p \to \C$, we define two operators: \emph{convolution} and \emph{correlation}:
$$
(f*g) (x) := \sum_{a + b = x} f(a) g(b) 
\quad \mbox{ and } \quad 
(f\circ g) (x) := \sum_{b - a = x} \ov{f(a)} g(b) \,.
$$
Fourier transform of the convolution and correlation might be written as follows:
\begin{equation}\label{f:F_svertka}
    \FF{f*g} = \FF{f} \FF{g}
    \quad 
    \quad \mbox{ and } 
    \quad \FF{f \circ g} = \ov{\FF{f}}\, \FF{g} \,.
\end{equation}

We can define Fourier transform of a set $A \subseteq \F_p$ via identifying 
set $A$ with its characteristic function $A: \F_p \to \{0,1 \}$.

\subsection{Wiener Norm}

The Wiener norm of a function $f : \F_p \to \C$ is defined as
\[
\wiener{f} = \frac{1}{p} \sum_{\xi \in \FF{\F}_p} |\FF{f} (\xi)|. 
\]
One can verify the following inequality for any two functions $f,g : \F_p \to \C$:
\begin{equation}\label{f:Wiener_mult}
    \wiener{fg} \le \wiener{f} \wiener{g}
\end{equation}

In this work, we will be interested in Wiener norms of \emph{sets} --- for a set $A \subseteq \F_p$, $\wiener{A}$ denotes the Wiener norm of its characteristic function. Given sets $P, Q$, the following properties of Wiener norm are well-known:
\begin{itemize}[itemsep=0.5em]
    \item If $P$ is nonempty, then $1 \le \wiener{P} \le \sqrt{|P|}$,

    \item If $P$ is an arithmetic progression, then $\wiener{P} \ll \log{p}$,

    \item $\wiener{P^c} = \wiener{P} + O(1)$,

    \item $\wiener{P \cap Q} \leqslant \wiener{P}\wiener{Q}$,

    \item If $P$ and $Q$ are disjoint, then $\wiener{P \sqcup Q} \leqslant \wiener{P} + \wiener{Q}$.
\end{itemize}

\subsection{Bounds on exponential sums}

At various points in the proof, we would like to verify that certain subsets of structured sets behave pseudorandomly in a strong quantitative sense. The following results will assist us with these.

First, we need the following variation of the classical Weil bound (see, e.g.,~\cite{CP_Stepanov, IK_book}).

\begin{authortheorem}[Weil]\label{t:Weil}
Let $\F_q$ be a finite field of characteristic $p$, and let $\psi$ be a nontrivial additive character on $\F_q$.
Let $f \in \F_q(x)$ be a rational function.
Suppose that $f(x)$ does not coincide with a function of the form $h(x)^p - h(x) + c$, where $h(x) \in \F_q(x)$ and $c \in \F_q$.
Then
\begin{equation}\label{f:Weil}
    \sum_{x \in \F_q \setminus \mathcal{P}} \psi(f(x)) \ll_f \sqrt{q},
\end{equation}
where $\mathcal{P}$ denotes the set of poles of $f$.
\end{authortheorem}

The second result is a famous theorem of Bombieri \cite[Theorem 6]{Bombieri_Weil_curve}. We present it in a simplified form (the definition of a pole on a curve can be found in \cite[Chapter 6, page 93]{Bombieri_Weil_curve}).

\begin{authortheorem}[Bombieri]\label{t:Bombieri}
Let $p$ be a prime number, $n\ge 2$ be a positive integer, $\mathcal{C} \subset \overline{\F}^n_p$ be an irreducible curve of degree $D$, $f\in \overline{\F}_p (x_1,\dots,x_n)$ be a rational function, $f=f_1/f_2$, where $f_1,f_2 \in \overline{\F}_p [x_1,\dots,x_n]$ have degrees $d_1,d_2$, and $\psi$ be a nontrivial additive character on $\F_p$. Suppose that $f$ is non--constant on $\mathcal{C}$. Then 
\begin{equation}\label{f:Bombieri}
\left| \sum_{(x_1,\dots, x_n) \in \mathcal{C}\setminus \mathcal{P}} \psi(f(x_1,\dots, x_n)) \right| 
\ll_{d_1,d_2,D} \sqrt{p},
\end{equation}
where $\mathcal{P}$ denotes the set of poles of $f$.
\end{authortheorem}

Since the classical curves below \eqref{def:C}, \eqref{def:C_Fermat} are absolutely irreducible, we obtain the following:

\begin{corollary}\label{corollary:Bombieri}
Let $p$ be a prime number and $g\in \F_p [x]$ be a polynomial of degree $s$ and $\mathcal{C}$ be the curve 
\begin{equation}\label{def:C}
    \mathcal{C} = \{ (x,y) \in \F_p \times \F_p ~:~ y^t = g(x) \} 
    \,,
\end{equation}
    where $(s,t)=1$ 
    or
\begin{equation}\label{def:C_Fermat}
    \mathcal{C} = \{ (x,y) \in \F_p \times \F_p ~:~ a x^s + b y^t = c \} 
    \,,
\end{equation}
    where all $s, t, a,b,c$ are non--zero. 
    Further let  $\psi$ be a nontrivial additive character on $\F_p$ and 
    suppose that $f\in \F_p (x,y)$ be a rational function, $f=f_1/f_2$, where $f_1,f_2 \in \F_p [x,y]$ have degrees $d_1,d_2$ such that $f$ is non--constant on $\mathcal{C}$. 
    Then 
\begin{equation}\label{cf:Bombieri}
    \left| \sum_{(x,y) \in \mathcal{C} \setminus \mathcal{P}} \psi(f(x,y)) \right| \ll_{d_1,d_2,s,t} \sqrt{p} \,,
\end{equation}
    where $\mathcal{P}$ denotes the set of poles of $f$.
\end{corollary}
\begin{proof}
    Any curve of the form \eqref{def:C} is absolutely irreducible; see, e.g., \cite[Chapter I]{Stepanov_book}.
    If we have a Fermat's curve of the form \eqref{def:C_Fermat}, then for $G(x,y) = a x^s + b y^t - c$ it is easy to see that the system $0=G(x,y) = \frac{\partial G}{\partial x} (x,y) = \frac{\partial G}{\partial y} (x,y)$ has no solution for $s,t\ge 1$. For negative $s,t$ use the projectivization. Thus, our curve is non-singular and, therefore, absolutely irreducible.
\end{proof}

\begin{corollary}\label{corollary:exp-conv}
Let the tuples $(r_1, \ldots, r_m), (d_1, \ldots, d_n)$ be generic and bounded.
Let
\[
    h(x) = \sum_{i=1}^{m} x^{r_i} + \sum_{j=1}^{n} (x - s)^{d_j},
\]
where $s \in \F_p$ is some fixed element.
Let $\psi$ be a nontrivial additive character on $\F_p$.

Then, assuming $h$ is not a constant,
\[
    \sum_{x \in \F_p} \psi(h(x)) \ll \sqrt{p}.
\]
\end{corollary}

\begin{proof}
Let $L$ be a scaling parameter, such that all $r_i' := r_i L, d_j' := d_j L$ are small modulo $p - 1$. Since it is coprime to $p - 1$, the map $x \mapsto x^L$ is a bijection of~$\F_p$.
Hence
\begin{multline*}
\sum_{x \in \F_p} \psi(h(x)) = 
\sum_{\substack{(y, z) :\\ y - z = s}} 
\psi\left( \sum_i y^{r_i} + \sum_j z^{d_j}\right) = \\ =
\sum_{\substack{(a, b):\\ a^L - b^L = s}}
\psi\left( \sum_i a^{Lr_i} + \sum_j b^{Ld_j}\right) =
\sum_{\substack{(a, b):\\ a^L - b^L = s}}
\psi\left( \sum_i a^{r_i'} + \sum_j b^{d_j'}\right)
\end{multline*}

By Corollary \ref{corollary:Bombieri} the curve $\mathcal{C} = \{ (a, b) \in \F_p \times \F_p ~:~ a^L - b^L = s \} $ is absolutely irreducible, $L = O(1)$, and therefore by Bombieri bound (\ref{f:Bombieri}),
\[
\sum_{x \in \F_p}\psi(h(x)) \ll \sqrt{p}.
\]
\end{proof}

Another deep and useful result is due to Bourgain (\cite{Bourgain_Mordell}; see also \cite{Konyagin_Mordell}).

\begin{authortheorem}[Bourgain]
Let $p$ be a prime, $k$ a positive integer, and $\d \in (0,1)$.
Then there exists $\eps = \eps(\d, k) > 0$ such that for any polynomial
\[
    f(x) = \sum_{j=1}^k a_j x^{d_j} \in \Z[x],
    \quad (a_j, p) = 1,
\]
whose exponents $1 \le d_j < p - 1$ satisfy
\[
    (d_j, p-1) < p^{1 - \d}, \quad
    (d_i - d_j, p-1) < p^{1 - \d} \ \text{for } i \neq j,
\]
one has
\begin{equation}
\label{eq:bourgain}
    \left| \sum_x e(f(x)) \right| \le p^{1 - \eps}.
\end{equation}
\end{authortheorem}

\subsection{Sumsets}

We  use representation function notations like $r_{A+B} (x)$ or $r_{A-B} (x)$ and so on, which count the number of ways $x \in \F_p$ can be expressed as a sum $a+b$ or $a-b$ with $a\in A$, $b\in B$, respectively. For example, $|A| = r_{A-A} (0)$.

Given two sets $A, B$, the sumset of $A$ and $B$ is defined in \eqref{def:A+B_intr}.
Now if $\eps \in (0,1)$, then  we write $A+_\eps B$ for the set of $x \in \F_p$ having at least $\eps p$ representations as a sum $a+b$, $a\in A$, $b\in B$.
Given positive integer $n \in \N$ and set $A \subseteq \Gr$ we write 
\[
nA := \underbrace{A + \ldots + A}_{\text{$n$ times}}.
\]
For arbitrary abelian group the Pl\"unnecke--Ruzsa inequality (see, e.g., \cite{Ruzsa_Plun}, \cite{TV}) holds, stating
\begin{equation}\label{f:Pl-R} 
|nA - mA| \le \left( \frac{|A+A|}{|A|} \right)^{n+m} \cdot |A|,
\end{equation} 
where $n,m$ are any positive integers. 
Further, if $|A+B|\le K|A|$ for some sets $A,B \subseteq \Gr$, then for any $n$ one has 
\begin{equation}\label{f:Pl-R+} 
    |nB| \le K^n |A| \,.
\end{equation}

Another useful set of inequalities is 
\[
|A||B - C| \leqslant |A - B||A - C|, \ \ \ |A||B + C| \leqslant |A + B||A + C|,
\]
where $A, B, C$ are arbitrary sets. 

Now, let $K$ be a parameter, and let $x_1, \ldots, x_k$ be any positive real values. We call them \emph{$K$-comparable} if 
\[
\max_i\{x_1, \ldots, x_k\} \ll K^{O(1)}\min_i \{x_1, \ldots, x_k \} \,.
\]

Let $P_1, \ldots, P_k$ be a collection of sets. We say that they are \emph{$K$-compatible}, if their cardinalities $|P_1|, \ldots, |P_k|$ are $K$-comparable, and for any choice of positive integers $a_1, \ldots, a_k \ll 1$ it holds that
\[
|a_1 P_1 + \ldots + a_k P_k| \ll K^{O_{a_1, \ldots, a_k}(1)}\max\{|P_1|, \ldots, |P_k|\}.
\]

The proof of the following proposition is a straightforward application of sumset inequalities above; we omit it.
\begin{proposition}\label{proposition:comparable-calculus}
Let $P_1, \ldots, P_k$ be sets of $K$-comparable sizes such that 
\[
|P_1 + \ldots + P_k| \ll K^{O(1)} \max{|P_1|, \ldots, |P_k|}.
\]
Let $\mathcal{F}_K$ be a family of arbitrary sums $a_1 P_1 + \ldots + a_k P_k$, further, their arbitrary subsets of size $\gg K^{-O(1)}\min\{|P_1|, \ldots, |P_k|\}$, as well as sumsets $X_i + P_i$ for $|X_i| \ll 1$, and, finally, their arbitrary unions. Then any finite collection of sets of $\mathcal{F}_K$ is $K$-compatible.
\end{proposition}

\subsection{Set unions/intersections}

When dealing with a union of sets, the following inequality is often useful:
\[
\left| S_1 \cup \ldots \cup S_n\right| \geqslant \sum_{i=1}^n |S_i| - \sum_{1\le i<j \le n} |S_i \cap S_j|.
\]
Under the additional hypothesis that all set sizes are at least $k$, and all set intersections are at most $l$, this results in the inequality $|S| \ge nk - \binom{n}{2}l$, which has the disadvantage that it does not improve with growth of $n$ beyond some threshold value. The following standard proposition helps in this situation. 

\begin{proposition}\label{proposition:set_intersect}
Let $S_1, S_2, \dots, S_n$ be finite sets, and let $k \geqslant l$ be positive integers, such that the following holds:
\begin{itemize}
    \item $|S_i| \geqslant k$ for all $i$,
    \item $|S_i \cap S_j| \leqslant l$ for all $i \neq j$.
\end{itemize}
Let $S := S_1 \cup S_2 \cup \dots \cup S_n$. Then
\[
|S| \geqslant \frac{nk^2}{k + nl}.
\]
\end{proposition}
\begin{proof}
    Let $m = \sum_{i\leqslant n}\sum_{s \in S}S_i(s) \geqslant nk$. 
    Observe that 
\begin{multline*}
   m^2 = \bigg(
    \sum_{s \in S}\big( \sum_{i=1}^{n} S_i(s) \big)
    \bigg)^2
    \leqslant 
    |S|\sum_{s \in S}\big( \sum_{i=1}^{n}S_i(s)\big)^2
    = |S|\sum_{s \in S}\sum_{i, j=1}^{n}S_i(s)S_j(s) 
    = \\
    = |S|\sum_{i, j \leqslant n}|S_i \cap S_j|
    \leqslant
    |S|\big(m + (n^2 - n)l\big) ,
\end{multline*}
    which implies
\[
    |S| \geqslant \frac{m^2}{m + (n^2 - n)l} \geqslant \frac{(nk)^2}{nk + (n^2 - n)l} \geqslant 
    \frac{nk^2}{k + nl}.
\]
\end{proof}

This has the following simple corollaries.

\begin{corollary}\label{corollary:set-intersect-k2l}
If $n \gg k/l$, then $|S| \gg k^2 / l$.
\end{corollary}

\begin{corollary}\label{corollary:set-intersect}
If all sets $S_i$ are supported in the ground set $G$ and $l|G| < k^2$, then
\[
n \leqslant \frac{|G| / k}{1 - \frac{l|G|}{k^2}}.
\]
\end{corollary}

We will often deal with intersections of sets. 
These sets can be approximated by other sets that are more suitable in some respects.
The following facts will assist us in maintaining good bounds on intersections of those approximations; the proof is a routine exercise.

\begin{proposition}\label{proposition:set-triangles}
Let $T_1, \ldots, T_k, U_1, \ldots, U_k$ be arbitrary sets. Then
\begin{itemize}[itemsep=0.3em]
\item 
$\left|(T_1 \cap \ldots \cap T_k)\ \triangle\ (U_1 \cap \ldots \cap U_k) \right| \leqslant |T_1\ \triangle\ U_1| + \ldots + |T_k\ \triangle\ U_k|$,
\item 
$\left|(T_1 \cup \ldots \cup T_k)\ \triangle\ (U_1 \cup \ldots \cup U_k) \right| \leqslant |T_1\ \triangle\ U_1| + \ldots + |T_k\ \triangle\ U_k|$,
\end{itemize}
\end{proposition}

The following proposition will be helpful when dealing with intersections of sets, having algebraic-independence property:
\begin{proposition}\label{proposition:prods-and-diffs}
    Let $T_1, \ldots, T_k, U_1, \ldots, U_k$ be arbitrary sets, supported on ground set $G$, such that $\max_i |T_i\ \triangle\ U_i| \leqslant \xi|G|$. Then 
    \[
    \frac{1}{|G|^{k - 1}}|T_1|\ldots|T_k| = \frac{1}{|G|^{k - 1}}|U_1|\ldots|U_k| + O(\xi |G|).
    \]
\end{proposition}
\begin{proof}
    Define $t_i = |T_i| / |G|, u_i = |U_i| / |G|$. Clearly, $0 \leqslant t_i, u_i \leqslant 1$, and $|t_i - u_i| \leqslant \xi$. The statement now is equivalent to $t_1\ldots t_k = u_1\ldots u_k + O_k(\xi)$, which follows from 
    \[
    t_1\ldots t_k - u_1\ldots u_k = 
    t_2\ldots t_k (t_1 - u_1) + u_1t_3\ldots t_k (t_2 - u_2) + \ldots + u_1\ldots u_{k - 1}(t_k - u_k).
    \]
\end{proof}

\section{Combinatorial properties of the Wiener norm}
\label{sec:comb-wiener}

In this section, we develop combinatorial tools for working with intersections of sets.  
In many of our arguments, we will require the following \emph{algebraic intersection property} (AIP), namely, the AIP takes place for some sets $P_1,\dots, P_k$ if  for generic exponents $r_1, \ldots, r_k$ one has 
\[
\bigl|\sqrt[r_1]{P_1} \cap \cdots \cap \sqrt[r_k]{P_k}\bigr|
= \frac{1 + o(1)}{p^{k-1}}|P_1|\cdots|P_k|.
\]
The AIP holds for a broad class of sets, which includes, in particular, sets with bounded Wiener norm.

\begin{lemma}\label{lemma:aip}
Let $W_1, \ldots, W_k$ be sets with Wiener norms bounded by $M$. Let tuple $(r_1, \ldots, r_k)$ be generic. Then
\[
\bigl|\sqrt[r_1]{W_1} \cap \cdots \cap \sqrt[r_k]{W_k}\bigr|
= \frac{1}{p^{\,k-1}}\,|W_1|\cdots|W_k|
  + O\left(M^{k} p^{1 - \eps}\right),
\]
where $\eps = \eps(k) > 0$.
\end{lemma}
\begin{proof}
Identifying sets $W_i$'s with their indicator functions, we write
\begin{equation}\label{eq:aip}
\bigl|\sqrt[r_1]{W_1} \cap \cdots \cap \sqrt[r_k]{W_k}\bigr|
= \sum_{x \in \F_p} W_1(x^{r_1}) \cdots W_k(x^{r_k}).
\end{equation}

Using the Fourier inversion formula we can write 
\[
W_i(x^r_i) = \frac{1}{p}\sum_{u_i}\FF{W}_i (u_i)e(u_i x^{r_i})
\]
for all $i = 1, \ldots, k$. Plugging it in we arrive at
\begin{multline*}
\bigl|\sqrt[r_1]{W_1} \cap \cdots \cap \sqrt[r_k]{W_k}\bigr| = 
\frac{1}{p^k}\sum_x\sum_{u_1, \ldots, u_k}
\FF{W}_1 (u_1)\ldots \FF{W}_k(u_k) e(u_1 x^{r_1} + \ldots + u_k x^{r_k}) = \\ = 
\frac{1}{p^k}\sum_{u_1, \ldots, u_k}
\FF{W}_1 (u_1)\ldots \FF{W}_k (u_k) \sum_x e(u_1 x^{r_1} + \ldots + u_k x^{r_k}).
\end{multline*} 
The tuple $(u_1, \ldots, u_k) = (0, \ldots, 0)$ contributes the main term $\frac{1}{p^{k - 1}}|W_1|\ldots|W_k|$.
By Bourgain's bound~\eqref{eq:bourgain}, the last sum has order $O(p^{1 - \eps})$, $\eps = \eps(k)$ if $(u_1, \ldots, u_k) \neq (0, \ldots, 0)$. Thus, 
\[
\bigl|\text{error term}| \ll 
\frac{1}{p^k}\sum_{u_1, \ldots, u_k}
|\FF{W}_1 (u_1)|\ldots |\FF{W}_k (u_k)| \times p^{1 - \eps} \leqslant M^k p^{1 - \eps}.
\]
\end{proof}

\begin{remark}
If tuple $(r_1, \ldots, r_k)$ is also bounded, then the error term is $O(\sqrt{p})$, which is guaranteed by Bombieri's exponential bound.
\end{remark}

Since finite intersections and complements of sets with bounded Wiener norm also have bounded Wiener norms, we can generate a lot of sets satisfying AIP. In practice, we would also like to have this property for sets, which can be \emph{approximated} by sets with small Wiener norms.

\begin{definition}
    We call a set $P$ \emph{approximable} if $|P| \gg p$ and there exists a set $W$ such that 
    \begin{itemize}[itemsep=0.5em]
        \item $|W\ \triangle\ P| = o(p)$ (small symmetric difference),
        \item $\wiener{W} = p^{o(1)}$ (bounded Wiener norm).
    \end{itemize}
    We denote the family of such sets by $\mathcal{F}$.
\end{definition}
\begin{remark}
\label{remark:F-levels}
    This definition is rather informal, as it does not specify the dependencies quantitatively. A possible solution would be to define the \emph{levels} $\mathcal{F}(c, \eps, M)$ consisting of sets $P, |P| \geqslant cp$ which might be approximated by sets $W$ such that $|W\ \triangle\ P| \leqslant \eps p$ and $\wiener{W} \leqslant M$. This would be a well-defined notion, making the bounds in Proposition \ref{proposition:F-closed} below explicit. However, in the context of this work, we have no interest in the particular expression of those dependencies, and we intentionally keep this definition informal for the sake of being concise.
\end{remark}

The family $\mathcal{F}$ provides a much broader source of sets with AIP:
\begin{proposition}
\label{proposition:aip}
    Any collection of sets $P_1, \ldots, P_k$ from the family $\mathcal{F}$ has AIP.
\end{proposition}
\begin{proof}
Let $(r_1, \ldots, r_k)$ be an arbitrary generic tuple, and let $W_1, \ldots, W_k$ be sets approximating $P_1, \ldots, P_k$. Then, for all $i$, the symmetric difference $\sqrt[r]{P_i} \triangle \sqrt[r]{W_i}$ is a subset of $\sqrt[r]{P_i \triangle W_i}$, and therefore its  size does not execeed $(r, p - 1)|P_i \triangle W_i| = o(p)$. Thus,
\begin{multline*}
\bigl|\sqrt[r_1]{P_1} \cap \cdots \cap \sqrt[r_k]{P_k}\bigr| 
\oversetgap{\text{Proposition~\ref{proposition:set-triangles}}}{=}
\bigl|\sqrt[r_1]{W_1} \cap \cdots \cap \sqrt[r_k]{W_k}\bigr| + o(p) = \\
\oversetgap{\text{Lemma~\ref{lemma:aip}}}{=} 
\frac{1}{p^{k-1}}|W_1|\ldots|W_k| + o(p) 
\oversetgap{\text{Proposition~\ref{proposition:prods-and-diffs}}}{=} 
\frac{1}{p^{k - 1}}|P_1|\ldots|P_k| + o(p) = \\
\oversetgap{|P_i| \gg p}{=}
\frac{1 + o(1)}{p^{k - 1}}|P_1|\ldots|P_k|.
\end{multline*}
\end{proof}

\medskip

There is a certain liberty one can take when working with this family:
\begin{proposition}\label{proposition:F-closed}
The family $\mathcal{F}$ has the following properties:
\begin{itemize}[itemsep=0.5em]
    \item[(1)] If $P_1, \ldots, P_k \in \mathcal{F}$, then $P_1 \cup \ldots \cup P_k \in \mathcal{F}$,

    \item[(2)] If $P \in \mathcal{F}$ and $|X| \ll 1$, then $X + P \in \mathcal{F}$,

    \item[(3)] Let $A, |A| = \alpha p$ be an arbitrary dense set, and let $\varepsilon \in (0,1)$.  
    There exists $D \in \mathcal{F}$ such that 
    \[
     A -_{\varepsilon\alpha^2} A \subseteq D \subseteq A - A.
    \]
\end{itemize}
 
\end{proposition}
Again, these inclusions are stated rather informally; see Remark \ref{remark:F-levels} above.

\medskip

The proof of Proposition \ref{proposition:F-closed} is given below, and for the proof of the last property (3), we will need to recall the results of our previous work.

Namely,  following~\cite{Semchankau_wrappers}, we introduce special subsets of $\F_p$ called \emph{wrappers}, each associated with two parameters: \emph{granularity}~$\varepsilon > 0$ and \emph{dimension}~$d \ge 1$.  
A wrapper $W$ with granularity~$\varepsilon$ and dimension~$d$ satisfies
\[
    \wiener{W} \le 
    \left(\frac{C \log p}{\varepsilon}\right)^{d},
\]
where $C > 0$ is an absolute constant.

\medskip

The following result, taken from~\cite{Semchankau_wrappers}, shows that level sets of convolutions can be efficiently approximated by wrappers.

\begin{namedlemma}[Supplementary]\label{lem:supp-wrapper}
Let $\alpha, \beta \in (0, 1)$ and $\delta, \xi > 0$.  
Let $A, B \subseteq \F_p$ with $|A| \ge \alpha p$ and $|B| \ge \beta p$, and let $\eta_1, \eta_2 \in \R$.  
Define
\[
    X := 
    \Bigl\{x \in \F_p : \tfrac{1}{p}(A * B)(x) \in [\eta_1, \eta_2]\Bigr\},
    \qquad
    X^{+} := 
    \Bigl\{x \in \F_p : \tfrac{1}{p}(A * B)(x) \in [\eta_1 - \delta,\, \eta_2 + \delta]\Bigr\}.
\]
Then there exist
\begin{itemize}
    \item a wrapper $W$ of granularity $\eps$ and dimension $d$ where
        \[
            \varepsilon = \varepsilon(\alpha, \beta, \delta)
            \gg \min\Bigl(\tfrac{\delta}{\sqrt{\alpha\beta}},\, 1\Bigr),
        \ \ 
            d = d(\alpha, \beta, \delta, \xi)
            \ll \max\Bigl(\tfrac{\alpha\beta}{\delta^2},\, 1\Bigr)
            \log\Bigl(\tfrac{1}{\xi}\Bigr),
        \]
    \item a set $Y \subseteq \F_p$ with $|Y| \le \xi p$,
\end{itemize}
such that
\[
    X \setminus Y \subseteq W,
    \qquad
    W \setminus Y \subseteq X^{+}.
\]
\end{namedlemma}

Now we are ready to prove Proposition \ref{proposition:F-closed}:
\begin{proof}
We will prove the properties one by one.
\begin{itemize}
    \item[(1)] Let $W_1, \ldots, W_k$ be the corresponding approximations. Let $P := P_1 \cup \ldots \cup P_k$ and $W := W_1 \cup \ldots \cup W_k$. We claim that $W$ approximates $P$. Indeed,
    \[
    |P\ \triangle\ W| 
    \oversetgap{\text{Proposition~\ref{proposition:set-triangles}}}{\leqslant} \sum_i |P_i\ \triangle\ W_i| = o(p),
    \]
    and 
    \[
    W \oversetgap{\text{Venn Diagram}}{=}
    \bigsqcup_{\{1, \ldots, k\} = I \sqcup J} \left( \bigcap_{i \in I} W_i \setminus \bigcap_{j \in J} W_j\right),
    \]
    implying that $\wiener{W} \leqslant (Cp^{o(1)})^k = p^{o(1)}$ for some absolute $C > 0$, since intersections, complements, and disjoint unions of sets with a bounded Wiener norm also have a bounded Wiener norm.

    \item[(2)] Let $W$ be the set approximating $P$. 
    Here we claim that $X + P$ is approximated by $X + W$. 
    Application of the second part of Proposition~\ref{proposition:set-triangles} with
    \[
        T_x := x + P, \qquad  U_x := x + W, \qquad x \in X,
    \]
    yields
    \[
     |(X + P)\ \triangle\ (X + W)| \le |X||P\ \triangle\ W| = o(p).
    \]
    Similarly to the previous point, we can write
    \[
    X + W = \bigsqcup_{X = Y \sqcup Z}
            \left(
            \bigcap_{y \in Y} (y + W)
            \setminus
            \bigcap_{z \in Z} (z + W)
            \right),
    \]
    and conclude that $X + W$ also has a bounded Wiener norm. 

    \item[(3)] Denote the set $\{x: (A \circ A)(x) \in [2\eps\alpha^2 p, p]\}$ by $D$, so that 
    \[
         A -_{\varepsilon\alpha^2} A \subseteq D \subseteq A - A.
    \]
    Let us choose a reasonably small parameter $\xi \in (0, 1)$. Application of the Supplementary Lemma to the correlation $A\circ A$ with parameters
    \[
        \eta_1 = \varepsilon \alpha^2, 
        \qquad
        \delta = \varepsilon \alpha^2/2,
        \qquad
        \eta_2 = 1, 
        \qquad 
        \eps = \eps, 
        \qquad 
        \xi = \xi
    \]
    gives a wrapper $W$ such that
    \[
        \wiener{W} \le (c \log p)^{d}, 
        \qquad
        |D\ \triangle\ W| \le \xi p,
    \]
    where
    \[
        c = c(\alpha, \varepsilon) \ll \frac{1}{\varepsilon \alpha},
        \qquad
        d = d(\alpha, \varepsilon, \xi)
            \ll \frac{\log(1/\xi)}{(\varepsilon \alpha)^2} \,.
    \]
    By choosing $\xi$ to be sufficiently small (but not too small) we can guarantee that $\wiener{W} = p^{o(1)}$ and $|D\ \triangle\ W| = o(p)$, so $W$ approximates $D$.
\end{itemize}
\end{proof}

\section{Proof of the Lower Bound}
\label{sec:lower-bound}

We begin this section by recalling the following result from our work~\cite{Semchankau_wrappers}.

\begin{namedlemma}[Wrapping]\label{lemma:qual-wrapping}
Let $k \geqslant 3$, and let dense sets $A_1, A_2, \ldots, A_k \subset \F_p$ be such that there are $o(p^{k-1})$ solutions to the equation
\[
    a_1 + \cdots + a_k = 0.
\]
Then there exist wrappers $W_1, \ldots, W_k$, and exceptional sets $Y_1, \ldots, Y_k \subset \F_p$ such that
\begin{enumerate}
    \item The equation $x_1 + \cdots + x_k = b$ has $o(p^{k-1})$ solutions in $x_i \in W_i$ for some $b \in \F_p$;
    
    \item $A_i \setminus Y_i \subseteq W_i$ and $|Y_i| = o(p)$ for all $i$;
    
    \item $\wiener{W_i} = p^{o(1)}$ for all $i$.
\end{enumerate}
\end{namedlemma}

The first property implies that
\[
    |W_1| + \cdots + |W_k| < p + o(p),
\]
see~\cite{Semchankau_wrappers} for details.

\medskip

We are now ready to prove Theorem~\ref{thm:lower-bound}.

\begin{proof}
Let $B$ denote the complement of the sumset 
$A^{r_1} + \cdots + A^{r_k}$.  
If $|B| = o(p)$, the result is immediate, so assume $|B| \gg p$ and set $\alpha := |A| / p, \beta := |B| / p$.

Applying the Wrapping Lemma to the sets 
$A^{r_1}, \ldots, A^{r_k}, B$, 
we obtain wrappers 
$W_1, \ldots, W_k, W_B$
and exceptional sets 
$Y_1, \ldots, Y_k, Y_B$ with $|Y_i| = o(p)$ 
and 
$A^{r_i} \setminus Y_i \subseteq W_i$.
Denote wrapper densities by $\omega_i := |W_i| / p$.

By removing at most $o(p)$ elements from $A$, we obtain a subset $A_0 \subseteq A$ of density $1 - o(1)$ such that
\[
    A_0^{r_i} \subseteq W_i 
    \quad\text{for all } i,
    \qquad\text{equivalently,}\qquad
    A_0 \subseteq 
    \sqrt[r_1]{W_1} \cap \cdots \cap \sqrt[r_k]{W_k}.
\]

By Lemma~\ref{lemma:aip},
\[
    |A_0| \leqslant 
    \bigl|\sqrt[r_1]{W_1} \cap \cdots \cap \sqrt[r_k]{W_k}\bigr|
    = \frac{1}{p^{k-1}}\,|W_1|\cdots|W_k| + o(p).
\]
Hence
\[
    \alpha \le \omega_1 \cdots \omega_k + o(1).
\]

On the other hand, from $\sum_i |W_i| < p + o(p)$ we deduce
\[
    \omega_B 
    \le 1 - (\omega_1 + \cdots + \omega_k) + o(1)
    \le 1 - k(\omega_1 \cdots \omega_k)^{1/k} + o(1)
    \le 1 - k \alpha^{1/k} + o(1).
\]

Finally, since $B \setminus Y_B \subseteq W_B$, we have
\[
    \beta \le \omega_B + o(1),
\]
and therefore
\[
    \beta \le 1 - k \alpha^{1/k} + o(1),
\]
as claimed.
\end{proof}

\section{Proof of the 99\% Theorem} 
\label{sec:99-percent}

We begin by recalling the well–known result of Green~\cite{Green_regularity}.

\begin{authortheorem}[Green, 2005]\label{thm:green}
Let $k \geqslant 3$, and let dense sets $A_1, A_2, \ldots, A_k \subseteq \F_p$ be such that there are $o(p^{k-1})$ solutions to the equation
\[
    a_1 + \cdots + a_k = 0.
\]
Then we may remove $o(p)$ elements from each $A_i$, leaving subsets $A_i' \subseteq A_i$ with the property that the equation
\[
    a_1' + \cdots + a_k' = 0, 
    \qquad a_i' \in A_i',
\]
has no solutions.
\end{authortheorem}

\medskip

\noindent
We are now ready to prove Theorem~\ref{thm:99-percent}, which would follow from a combination of the Wrapping Lemma, Green’s theorem, and the basic properties of wrappers. The proof begins where the proof of Theorem \ref{thm:lower-bound} ends.

\medskip

\begin{proof}
Let the sets $B, A_0, W_i, Y_i$ ($i = 1, \ldots, k$) be as defined in the proof of Theorem~\ref{thm:lower-bound}.  
By the first property guaranteed by the Wrapping Lemma, the equation
\[
    x_1 + \cdots + x_k = x_B, 
    \qquad x_i \in W_i, \ x_B \in W_B - b,
\]
has only $o(p^{k-1})$ solutions for some $b \in \F_p$.  
By Green’s theorem, we may remove $o(p)$ elements from each of 
$W_1, \ldots, W_k, (W_B - b)$ 
to obtain sets 
$P_1, \ldots, P_k, P_B$ (absorbing the shift into notation)
such that the equation
\[
    x_1 + \cdots + x_k = x_B
\]
has no solutions with $x_i \in P_i$ and $x_B \in P_B$.

\medskip

In particular, the sumset $P_1 + \cdots + P_k$ is disjoint from $P_B$.  
Using this and $|B| \le (1 + o(1))|W_B|$, we get
\begin{align*}
    |P_1 + \cdots + P_k|
    &\le p - |P_B| =\\
    &= p - (1 + o(1))|W_B| \le \\
    &\le p - (1 + o(1))|B| = (1 + o(1))\, K\, p^{\frac{k-1}{k}} |A|^{1/k}.
\end{align*}

\medskip

Next, since $P_i$'s are approximated by $W_i$, we have
\[
\left|\sqrt[r_1]{P_1} \cap \cdots \cap \sqrt[r_k]{P_k}\right|
\oversetgap{\text{Proposition~\ref{proposition:aip}}}{=}
\frac{1}{p^{k-1}}\,|P_1|\cdots|P_k| + o(p).
\]

Therefore, we obtain the desired result for $A_0 \subseteq A$ (resulted from removing at most $o(p)$ elements) with $P_i, i = 1, \ldots, k$.
\end{proof}

\section{On an additive  distance between sets}
\label{sec:new_norm}

In this section we discuss a new distance between additive sets and obtain some applications of these concepts to our problem about $A+A^*$ in the second part.

\subsection{Additive distance}
\label{ssec:new_normI}

\begin{definition}
\label{def:norm}
    Let $f,g : \Gr \to \C$ be functions. 
    We write 
\begin{equation}\label{def:rho}
    \rho^2 (f,g) = \| f \circ  f - g \circ  g\|_\infty 
    \quad \quad
    \mbox{ and }
    \quad \quad
    \rho^2_* (f,g) = \max_{x\neq 0} | (f \circ  f) (x) - (g \circ  g) (x)| \,.
\end{equation}
    We  put $\rho(f)$ for $\rho (f,0) = \| f \circ  f\|^{1/2}_\infty$, and similarly for $\rho_*(f)$. 
\end{definition}


Let us make 
some simple remarks concerning the definition above and compare $\rho(f,g)$ with some standard metrics. 
Other relations are discussed in the appendix. 
First of all, it is easy to see that $\rho(\cdot, \cdot)$, $\rho_* (\cdot, \cdot)$ are  pseudometrics (for example, one has $\rho(f,g) = \rho(f+c,g+c)$ for any $c$ and arbitrary functions $f,g$ with zero mean).
Let us remark that  $\| f\|_2 \le \rho(f)$. 
Also, if $\d:= N^{-1} \sum_x f(x)$, then 
$\rho(f,\d) = \rho(f- \d)$. 
Finally,  there is a simple relation between 
the quantity $\rho (f)$
and Fourier transform. 
Indeed, 
using 
\eqref{f:Fourier_representations}
one has 
\begin{equation}\label{f:Fourier_*}
    | \FF{f} (\xi)|^2 = | \sum_{x} (f \circ f) (x) \ov{\xi (x)} | \le \rho^2(f) N 
\end{equation}
and hence $\| \FF{f}\|_\infty \le \rho(f) \sqrt{N}$.


\subsection{The family $\mathcal{W}$}
\label{ssec:new_normII}

In this subsection, we study the family of sets $\mathcal{W}$ (not to be  confused with the family in Proposition \ref{proposition:comparable-calculus}), which will play a central role in the next section. In general terms, the elements of this family can be viewed as intersections of polynomial images of arithmetic progressions.

\begin{definition}
\label{def:family_W}
Let $p$ be a prime number, $k$ a positive integer, and $M \geqslant 1$ a real number. 
A function $f : \F_p \to \C$ belongs to the family $\mathcal{W}(k, M)$ if it has the form
\[
    f(x) = f_{S_0, \dots, S_k, r_1, \ldots, r_k}(x)
    = S_0(x) \prod_{i=1}^k S_i(x^{r_i}),
\]
where the sets $S_i$ satisfy $\wiener{S_i} \le M$ for all $i = 0, 1, \dots, k$, and the tuple $(1, r_1, \ldots, r_k)$ is coprime. Again, sets $S_i$'s are identified with their characteristic functions.
We write $S_j(f) = S_j$, further 
\[
    \d_i(f) := p^{-1} \sum_x S_i(x), \quad i = 1, \ldots, k \,,
\]
and set $\d_f := \prod_{i=1}^k \d_i(f)$.
\end{definition}
Note that by the multiplicative property of the Wiener norm~\eqref{f:Wiener_mult}, we have
\[
    \mathcal{W}(n, M) \cdot \mathcal{W}(m, M) \subseteq \mathcal{W}(n + m, M^2).
\]

We now state the main lemma regarding the family $\mathcal{W}(k, M)$. 
It shows that any function $f \in \mathcal{W}(k, M)$ is close, in the $\rho(\cdot, \cdot)$-distance, to its \emph{mean component} $\d_f \cdot S_0(f)$.

\begin{lemma}
\label{lemma:convolutions}
Let 
\[
    f = f_{S_0, \ldots, S_n, r_1, \ldots, r_k} \in \mathcal{W}(k, M),
\]
where the tuple $(1, r_1, \ldots, r_k)$ is coprime and bounded.
Then 
\begin{equation}
    \rho_* (f, \d_f S_0(f)) \ll M^{2k + 2} \sqrt{p}.
\end{equation}
\end{lemma}

\begin{proof}
Since all $r_i$ are coprime to $p - 1$, we conclude that they are odd, and therefore at least $2$-separated. It is an easy exercise then, that for all $s \neq 0$ the polynomials $1, x, x^{r_1}, \ldots, x^{r_k}, (x-s)^{r_1}, \ldots, (x - s)^{r_k}$ are linearly independent in $\F_p[x]$.

We write
\[
    f(x) = \prod_{i=0}^k S_i(x^{r_i})
\]
where $r_0 = 1$.
By the inverse formula~\eqref{f:inverse},
\[
    S_i(x^{r_i}) = p^{-1} \sum_{u_i} \widehat{S}_i(u_i) e(u_i x^{r_i}),
\]
hence
\[
    f(x) = p^{-(k+1)} \sum_{u_0, \dots, u_n}
        \Bigl( \prod_{i=0}^k \widehat{S}_i(u_i) \Bigr)
        e\left( \sum_{i=0}^k u_i x^{r_i} \right).
\]
Substituting these expressions into the formula of $(f \circ f)(s)$ for any $s \in \F^*_p$ gives
\begin{multline*}
    (f \circ f)(s) = 
    \sum_{x - y = s} f(x) f(y) = \\ =
    p^{- (2k + 2)} 
    \sum_{x - y = s}
    \sum_{u_0, \dots, u_k}
    \sum_{v_0, \dots, v_k}
    \prod_{i=0}^k \widehat{S}_i(u_i)
    \prod_{j=0}^k \widehat{S}_j(v_j) 
    \times
    e\left( 
        \sum_{i=0}^k u_i x^{r_i} 
        + \sum_{j=0}^k v_j y^{d_j}
    \right)  = \\ =
    p^{- (2k + 2)} 
    \sum_{u_0, \dots, u_k}
    \sum_{v_0, \dots, v_k}
    \prod_{i=0}^k \widehat{S}_i(u_i)
    \prod_{j=0}^k \widehat{S}_j(v_j) 
    \times
    \sum_x
    e\left( 
        \sum_{i=0}^k u_i x^{r_i} 
        + \sum_{j=0}^k v_j (x - s)^{r_j}
    \right).
\end{multline*}

The main term $\sigma_*(s)$ arises when $u_i = v_i = 0$ for $i \ge 1$.
In this case, the inner sum over $x$ becomes 
\[
e(-v_0s) \sum_x e\Bigl((u_0 + v_0)x\Bigr)
\]
and it is nonzero only when $u_0 = -v_0$.
Thus, using~\eqref{f:Fourier_representations},
\begin{multline*}
    \sigma_*(s)
    = p^{-(2k + 1)} 
    \prod_{i=1}^k \Bigl( \sum_z S_i(z) \Bigr)
    \prod_{j=1}^k \Bigl( \sum_w S_j(w) \Bigr)
    \sum_u \widehat{S}_0(u)\widehat{S}_0(-u) e(us) = \\
    = \d_f^2 \cdot (S_0 \circ S_0)(s)
    = \d_f^2 \cdot \bigl(S_0(f) \circ S_0(f)\bigr)(s).
\end{multline*}

Let $\mathcal{E}(s) := (f \circ f)(s) - \sigma_*(s)$ denote the error term. Then, expanding as before, we obtain
\[
    \mathcal{E}(s) =
    p^{ -(2k + 2)} 
    \sum_{u_0, \dots, u_k}
    \sum_{v_0, \dots, v_k}
    \prod_{i=0}^k \widehat{S}_i(u_i)
    \prod_{j=0}^k \widehat{S}_j(v_j) 
    \times
    \sum_x
    e\left( 
        \sum_{i=0}^k u_i x^{r_i} + \sum_{j=0}^k v_j (x - s)^{r_j}
    \right).
\]
where not all $u_i, v_j$ vanish for $i, j \ge 1$.
By Corollary \ref{corollary:exp-conv}, each exponential sum has order $O(\sqrt{p})$. Hence
\begin{multline*}
    \mathcal{E}(s)
    \ll
    p^{-(2k + 2)}
    \sum_{u_0, \dots, u_k}
    \sum_{v_0, \dots, v_k}
    \prod_{j=0}^k |\widehat{U}_j(u_j)|
    \prod_{i=0}^k |\widehat{V}_i(v_i)|
    \cdot 
    \sqrt{p} \ll 
    M^{2k + 2}\sqrt{p}.
\end{multline*}
This completes the proof.
\end{proof}

\begin{remark}
    This proof can be easily altered to approximate functions 
    \[f \circ g, \quad f * g,\] by 
    \[\d_f \d_g \cdot S_0(f) \circ S_0(g), \quad \d_f \d_g \cdot S_0(f) * S_0(g),\]
    respectively.
\end{remark}

At the end of this subsection, consider another family of sets which can be approximated by its subsets in the sense of distance $\rho(\cdot, \cdot)$.

\begin{remark}
    Consider the family $\mathcal{R}$ of functions $A_S (x) = S(x) R(x)$, where $R$ is the set of quadratic residues and $\wiener{S} \le M$.
    Using similar computations as in the proof of Lemma \ref{lemma:convolutions} and the Weil bound, one can show 
\[
    \| A_S * A_T - \d_{A_S} \d_{A_T} \cdot S * T \|_\infty \ll M^2 \sqrt{p} \,.
\]
    However, the family $\mathcal{R}$ is not invariant under additive shifts, and a complete analogue of the Lemma \ref{lemma:convolutions} for the products (consider the set $A_S (x) A_T (x+1)$, say) requires an analogue of the Weil estimate over some surfaces.
\end{remark}

\section{Sumsets with pseudorandom sets}
\label{sec:sumsets}
To motivate the main results of this section, let us imagine that we work with a sumset $A + A^*$.

From the proof of the \emph{99\% Theorem}, we know that there exist sets $W_1, W_2$ with Wiener norms of order $p^{o(1)}$, and subsets 
$P \subseteq W_1$, $Q \subseteq W_2, A_0 \subseteq A$ of relative densities $1 - o(1)$ such that
\begin{itemize}[itemsep=0.5em]
    \item $|P + Q| \leqslant (1 + o(1)) K \sqrt{p|A|}$;
    \item $|P \cap Q^*| = \frac{1+o(1)}{p}|P||Q|$;
    \item $A_0 \subseteq P \cap Q^*$.
\end{itemize}

Let $\kappa_1 = |P| / p$, $\kappa_2 = |Q| / p$, and density of $A_0$ in $P \cap Q^*$ be $\eta$.
The last two bulletpoints imply
\[
    \alpha \leqslant (1 + o(1)) \kappa_1 \kappa_2,
\]
and the first one gives
\[
    |P + Q| \leqslant (1 + o(1)) K \sqrt{\eta}   \sqrt{|P||Q|}.
\]
Since $|P + Q| \geqslant \max(|P|, |Q|)$, it follows that
\[
    \eta \gg 1 / K^2, 
    \qquad 
    \kappa_1, \kappa_2 \asymp_K \sqrt{\alpha}.
\]

From these relations and the inequality $|A + A^*| \leqslant K \sqrt{p|A|}$ we deduce
\[
    |A + A_0^*| \ll_K |P|.
\]
Intuitively, the set $P \cap Q^*$ should behave similarly to a random subset of $P$, and $A_0$ should behave similarly to a random subset of some $P_0 \subset P$ of density $\eta$. Hence, one might expect the approximation
\begin{equation}\label{eq:expectation}
    A^* + A_0 \approx A^* + P_0,
\end{equation}
to hold. If~\eqref{eq:expectation} were true, then
\[
    |A^* + P_0| \approx |A^* + A_0| \ll_K |P| \ll_K |P_0|
\]
would hold, and by the Ruzsa Covering Lemma one obtains
\[
    A^* \subseteq X_P + (P_0 - P_0),
    \qquad 
    |X_P| \ll_K 1.
\]
Applying the same argument to $A + A_0^*$ yields
\[
    A \subseteq X_Q + (Q_0 - Q_0),
\]
and therefore
\[
    A \subseteq (X_Q + (Q_0 - Q_0)) \cap (X_P + (P_0 - P_0))^*,
\]
an intersection of precisely the form we expect.

\medskip

Our goal 
now 
is to 
arrive  to 
a similar conclusion \emph{without} assuming~\eqref{eq:expectation}.
We will proceed as follows:
\begin{itemize}
    \item Although we cannot directly obtain the approximation $A^* + A_0 \approx A^* + P_0$, we can deduce its consequence $|A^* + A_0| \gg_K |P|$. 
    This is based on pseudorandom properties of the set $P \cap Q^* \subset P$. 

    \item The main idea is to treat $P$ as if it were an additive subgroup of $\F_p$, and to consider a ``quotient'' of $\F_p$ modulo $P$ and get almost-disjoint ``cosets''.
    Using the previous result, we show that
    \[
        |A^* + A_0| \gg_K |P| \cdot \#\{\text{cosets}\},
    \]
    which contradicts the upper bound $K^{O(1)} |P|$, if $A$ has large intersections with many such cosets.
    This part of the argument is purely combinatorial.
\end{itemize}

These ideas generalize to arbitrary bounded coprime exponent configurations at a small additional cost. 
That is, we will prove the following lemma:

\begin{lemma}
\label{lemma:X+T-growth-structure}
Let $W, W_1, \ldots, W_k \subseteq \F_p$ be sets with Wiener norms $p^{o(1)}$
and densities $\omega, \omega_1, \ldots, \omega_k$, respectively.
Set $\omega_\times := \omega_1 \cdots \omega_k$.
Let tuple $(r_1, \ldots, r_k)$ be coprime and bounded.

Let $P \subseteq W$ and $P_i \subseteq W_i$ $(i=1,\ldots,k)$ be subsets of relative density at least $1-\varepsilon$,
where $\varepsilon \ll \omega \omega_\times$.
Define
\[
    R := P \cap \sqrt[r_1]{P_1} \cap \cdots \cap \sqrt[r_k]{P_k},
\]
and let $T \subseteq R$ be a subset of relative density $\eta$.
Let $Y \subseteq \F_p$ be arbitrary.

\begin{enumerate}
    \item[\textup{(a)}] If $|Y| \gg \eta/\omega_{\times}$, then
    \[
        |Y+T| \gg \eta^2 |P|.
    \]

    \item[\textup{(b)}]
    Assume in addition that $|P - P| \le L|P|$.
    If
    \[
        |Y + T| \le K|P|,
    \]
    then there exist sets $X, E \subseteq \F_p$ such that
    \[
        Y \subseteq \bigl(X + (P -_{\kappa^2/2} P)\bigr) \sqcup E,
        \qquad
        |X| \ll \frac{KL^7}{\eta^2},
        \qquad
        |E| \ll \frac{1}{\omega \omega_\times},
    \]
    where $\kappa := |P|/p$.
\end{enumerate}
\end{lemma}

\begin{remark}
    In applications, $\eta$ has order $K^{-O(1)}$ and is `small', while $\omega_{\times}$ has order $\alpha^{O(1)}$ and is `very small'. 
\end{remark}

\begin{remark}
    Although we will only apply this result for \emph{dense} sets $P_i, W_i$, it is worth noting that it works for sets with cardinalities as small as $\Omega(p^{1 - c}), c > 0$.
\end{remark}

\subsection{Proof of Lemma~\ref{lemma:X+T-growth-structure}\textup{(a)}}

The main idea is that $R$ behaves similarly to a pseudorandom subset of $P$.
The next two propositions make this precise; both are simple consequences of
Lemma~\ref{lemma:convolutions}. Define
\[
    I := W \cap \sqrt[r_1]{W_1} \cap \ldots \cap \sqrt[r_k]{W_k}.
\]

\begin{proposition}\label{proposition:I-I}
Assume $\wiener{W}, \wiener{W_1}, \ldots, \wiener{W_k} \le M$.
Then for any $\lambda \neq 0$,
\[
    r_{I-I}(\lambda)
    = \omega_\times^{2}  r_{W-W}(\lambda)
    + O\!\left(M^{2k+2}\sqrt{p}\right).
\]
\end{proposition}
\begin{remark}
    Here we really need that exponents $r_i$ are coprime to $p - 1$, which is evident from the following example with $r_1 = 2$: 
    \[
    W := \{x: \omega_1 p + 10 \leqslant x \leqslant (\omega_1 + \omega) p + 10\}, \quad 
    W_1 := \{x: 0 \leqslant x \leqslant \omega_1 p\}.
    \]
    It is clear from the definition that sets $W + W$ and $W_1 - W_1$ are disjoint (provided $\omega_1, \omega_2$ are sufficiently small). It is also clear, that $r_{W - W}(1) = |W| - 1 \approx \omega p$. On the other hand, $r_{I - I}(1) = 0$, where $I:= W \cap \sqrt{W_1}$ and this is far from $\omega_1^2 r_{W - W} (1)$. Indeed, if there was $x \in I$ such that the element  $y = x + 1$ is also in $I$, we would obtain $y^2 - x^2 = (y - x)(y + x) = y + x \in W + W$. On the other hand, $y^2 - x^2 \in W_1 - W_1$, which is a contradiction. 

\end{remark}

\begin{proof}
Treat $I$ as its indicator function. Since $I \in \mathcal{W}(k,M)$,
Lemma~\ref{lemma:convolutions} gives
\[
    \bigl|r_{I-I}(\lambda) - \omega_\times^{2} r_{W-W}(\lambda)\bigr|
    \ll M^{2k+2}\sqrt{p},
\]
as required.
\end{proof}

\begin{proposition}\label{proposition:R-R}
For any $\lambda \neq 0$ one has
\[
    r_{R-R}(\lambda)
    \le
    \omega_\times^{2}  r_{P-P}(\lambda)
    + O\!\left(\varepsilon \omega \omega_\times^{2}p\right).
\]
\end{proposition}

\begin{proof}
Since $R \subseteq I$, we have $r_{R-R}(\lambda) \le r_{I-I}(\lambda)$.
By Proposition~\ref{proposition:I-I},
\[
    r_{I-I}(\lambda)
    =
    \omega_\times^{2} r_{W-W}(\lambda)
    + O\!\left(M^{2k+2}\sqrt{p}\right).
\]
Moreover, since $P \subseteq W$ and $|W\setminus P|\le \varepsilon |W|$,
\[
    r_{W-W}(\lambda) \le r_{P-P}(\lambda) + 2|W\setminus P|
    \le r_{P-P}(\lambda) + 2\varepsilon|W|
    = r_{P-P}(\lambda) + O(\varepsilon \omega p).
\]
Combining these estimates and using $M=p^{o(1)}$ yields
\[
    r_{R-R}(\lambda)
    \le
    \omega_\times^{2} r_{P-P}(\lambda)
    + O\!\left(\varepsilon \omega \omega_\times^{2}p\right),
\]
as claimed.
\end{proof}

\medskip

We also record the size of $R$ (and hence of $T$).
On the one hand, by the algebraic intersection property,
\[
    |R| \le |I|
    = \bigl|W \cap \sqrt[r_1]{W_1}\cap\cdots\cap \sqrt[r_k]{W_k}\bigr|
    = (1+o(1)) \omega \omega_\times p.
\]
On the other hand, since $|W_i\setminus P_i|\le \varepsilon |W_i|$ for all $i$,
\[
    |R|
    \ge |I| - |W\setminus P| - |W_1 \setminus P_1| - \ldots -|W_k \setminus P_k|
    = |I| - O(\varepsilon p),
\]
so $|R| \sim \omega \omega_\times p \sim \omega_\times |P|$ and $|T| \sim \eta |R| \sim \eta \omega_\times |P|$.

\medskip
\noindent
\emph{Proof of Lemma~\ref{lemma:X+T-growth-structure}(a).}
For distinct $y, y'$ we have
\begin{multline*}
    |(y + T) \cap (y' + T)| 
    = r_{T - T}(y - y')
    \leqslant r_{R - R}(y - y') \leqslant \\
    \leqslant \omega_\times^2 r_{P - P}(y - y') + O(\eps \omega \omega_\times^2 p) 
    \leqslant \omega_\times^2 |P| + O(\eps \omega \omega_\times^2 p).
\end{multline*}

Set
\[
    k := |T| \sim \eta \omega_\times |P|,
    \qquad
    \ell := \max_{i \neq j} |(y + T) \cap (y' + T)| 
        \ll \omega_\times^2 |P|.
\]
Since all translates $y + T$ are supported in $Y + T$, and $|Y| \gg k / \ell \sim \eta / \omega_\times$, we can apply Corollary~\ref{corollary:set-intersect-k2l} to obtain
\[
    |Y + T| \gg \frac{k^2}{\ell} \sim \eta^2 |P|.
\]
This completes the proof.
\qed

\subsection{Proof of Lemma~\ref{lemma:X+T-growth-structure}(b)}

Let $P \subseteq \F_p$ be a set with small doubling, i.e.\ $|P - P| \leqslant K|P|$.
It is well known (see, e.g., \cite[Section~2.4]{TV}) that the entire group $\F_p$
can be covered by a bounded number of translates of $P - P$ via the following greedy algorithm:
\begin{enumerate}
    \item Initialize $X := \varnothing$.
    \item At each step, if there exists $x \in \F_p$ such that
    $x + P$ is disjoint from all $x' + P$ with $x' \in X$, add $x$ to $X$.
    \item When no such $x$ exists, conclude that
    \[
        \F_p = \bigcup_{x \in X} (x + P - P).
    \]
\end{enumerate}
Clearly,
\[
    |X| \leqslant \frac{p}{|P|} \leqslant \frac{pK}{|P - P|}.
\]

\medskip

Unfortunately, even if $P$ has small doubling, the difference set $P - P$
does not necessarily possess the algebraic intersection property required in our arguments.
However, as shown in Section~\ref{sec:comb-wiener}, there exists a subset
$D \subseteq P - P$ that \emph{does} have this property and, moreover,
contains a set of popular differences.

A simple modification of the above algorithm then yields a covering by a bounded number
of translates of the popular difference set.
In this modified version, we allow the translates $x + P$ to have small intersections.
Since the trivial bound
$p \geqslant |X + P| = |X||P|$
is no longer available, we will use Corollary~\ref{corollary:set-intersect}
to control the size of $|X|$.

\begin{proposition}\label{proposition:popular-covering}
Let $P \subseteq \F_p$ have size $\kappa p$, and let $\varepsilon \in (0,1)$.
Then there exists a set $X \subseteq \F_p$ such that
\[
    \F_p = \bigcup_{x \in X} \bigl(x + P -_{\varepsilon \kappa^2} P \bigr),
    \qquad
    |X| \leqslant \frac{1}{\kappa(1 - \varepsilon)}.
\]
Moreover, the translates $x + P$ are almost disjoint: for any $x, x' \in X$,
\[
    |(x + P) \cap (x' + P)| \leqslant \varepsilon \kappa^2 p.
\]
\end{proposition}

\begin{proof}
The modified algorithm proceeds as follows:
\begin{enumerate}
    \item Initialize $X := \varnothing$.
    \item At each step, if there exists $x \in \F_p$ such that
    \[
        |(x + P) \cap (x' + P)| \leqslant \varepsilon \kappa^2 p
        \quad \text{for all } x' \in X,
    \]
    then add $x$ to $X$.
    \item When no such $x$ exists, we claim that
    \[
        \F_p = \bigcup_{x \in X} (x + P -_{\varepsilon \kappa^2} P).
    \]
\end{enumerate}

To verify the claim, let $x \in \F_p$ be arbitrary.
Since $x \notin X$, there exists some $x' \in X$ such that
\[
    |(x + P) \cap (x' + P)| \geqslant \varepsilon \kappa^2 p.
\]
Thus there are at least $\varepsilon \kappa^2 p$ pairs $(p_1,p_2) \in P \times P$
with $x + p_1 = x' + p_2$, i.e. $x = x' + (p_2 - p_1)$. Hence $x \in x' + P -_{\varepsilon \kappa^2} P$, which proves the covering.

Finally, all sets $x + P$ have size $\kappa p$, and their pairwise intersections
are bounded by $\varepsilon \kappa^2 p$.
Applying Corollary~\ref{corollary:set-intersect} with $G = \F_p$ gives $|X| \leqslant \frac{1}{\kappa(1 - \varepsilon)}$.
\end{proof}

As a byproduct, we obtain
\[
    |P -_{\varepsilon \kappa^2} P| \geqslant (1 - \varepsilon)|P|.
\]
Although stronger results are known (see references in \cite{Semchankau_wrappers}),
it suffices for our purposes that
\begin{equation}\label{eq:byproduct}
    |P -_{\varepsilon \kappa^2} P| \gg |P|
    \qquad \text{for all } \varepsilon \in (0,1).
\end{equation}

\medskip

\noindent
\emph{Proof of Lemma~\ref{lemma:X+T-growth-structure}\textup{(b)}.}
Apply Proposition~\ref{proposition:popular-covering} with $\varepsilon = 1/2$ to the set $P$.
This yields a set $X$, with $|X| \leqslant 2/\kappa$, such that $\F_p$ is covered by the translates
\[
    x + P -_{\kappa^2/2} P, \qquad x \in X,
\]
and any two translates $x + P$, $x' + P$ intersect in at most $\kappa^2 p/2$ elements.

For each $x \in X$ define
\[
    Y_x := Y \cap \bigl(x + P -_{\kappa^2/2} P\bigr).
\]
Clearly,
\[
    Y = \bigcup_{x \in X} Y_x.
\]
Let
\[
    X' := \{\, x \in X : |Y_x| \gg \eta / \omega_\times \,\},
    \qquad
    E := \bigcup_{x \in X \setminus X'} Y_x.
\]
Then
\[
    |E|
    \leqslant |X| \max_{x \in X \setminus X'} |Y_x|
    \ll \frac{1}{\kappa} \cdot \frac{\eta}{\omega_\times}
    \ll \frac{\eta}{\omega\,\omega_\times}.
\]

\smallskip

We now define a graph with vertex set $X'$, where $x \sim x'$ if the sets
\[
    x + P - P + P
    \quad \text{and} \quad
    x' + P - P + P
\]
intersect.
Fix $x \in X'$. If $x' \sim x$, then
\[
    x' \in x + 3P - 3P,
\]
and hence
\[
    x' + P \subseteq x + 4P - 3P =: G_x.
\]
By the Plünnecke--Ruzsa inequality,
\[
    |G_x| = |4P - 3P| \leqslant \min(L^7 |P|,\, p).
\]

Since the translates $x' + P$ are almost disjoint,
Corollary~\ref{corollary:set-intersect}, applied with ground set $G = G_x$,
sets $\{x' + P : x' \sim x\}$,
and parameters $k = \kappa p$, $l = \kappa^2 p/2$, gives
\[
    \deg(x) = \#\{x' : x' \sim x\} \leqslant 2L^7.
\]
Here $\deg (x)$ is the degree of a vertex $x\in X'$ in our graph. 
It is well known that a graph on $n$ vertices with maximal degree $d$
contains an independent set of size at least $n/(d+1)$.
Therefore, there exists an independent set
\[
    X'' \subseteq X'
\]
of cardinality at least $|X'|/(2L^7 + 1)$, and the sets
\[
    x'' + P - P + P, \qquad x'' \in X'',
\]
are pairwise disjoint.

\smallskip

Now let $T \subseteq P$ satisfy $|Y + T| \leqslant K|P|$.
For each $x'' \in X''$ we have
\[
    Y_{x''} + T
    \subseteq (x'' + P - P) + T
    \subseteq x'' + P - P + P,
\]
and these sets are disjoint.
Consequently,
\[
    K|P|
    \geqslant |Y + T|
    \geqslant \sum_{x'' \in X''} |Y_{x''} + T|
    \overset{\text{Lemma~\ref{lemma:X+T-growth-structure}\textup{(a)}}}{\gg}
    |X''| \cdot \eta^2 |P|
    \gg \frac{|X'|}{L^7}\,\eta^2 |P|.
\]
Thus $|X'| \ll KL^7/\eta^2$.

Finally,
\[
    Y
    = \Bigl(\bigcup_{x \in X'} Y_x\Bigr) \cup E
    \subseteq
    \Bigl(\bigcup_{x \in X'} \bigl(x + P -_{\kappa^2/2} P\bigr)\Bigr) \cup E,
\]
which completes the proof.
\qed

\section{Proof of the $100\%$ Theorem}
\label{sec:100-percent}

Let $A_0, W_i, P_i$ for $i = 1, \ldots, k$ be as in the proof of
Theorem~\ref{thm:99-percent}.
In particular, recall that
\[
    \frac{|A_0|}{|A|} \geqslant 1 - o(1),
    \qquad
    \frac{|P_i|}{|W_i|} \geqslant 1 - o(1),
    \qquad
    A_0^{r_i} \subseteq P_i,
    \quad i = 1, \ldots, k.
\]

Let $\omega_1, \ldots, \omega_k$ denote the densities of $W_i$.
We recall that
\[
    (1 + o(1)) \alpha
    \leqslant \omega_1 \cdots \omega_k
    \leqslant \left(\frac{\omega_1 + \cdots + \omega_k}{k}\right)^k
    \leqslant (1 + o(1)) (K/k)^k \alpha.
\]
In particular, $\omega_i \sim K^{O(1)} \alpha^{1/k}$ for all $i$.
Thus the densities $\omega_i$, and hence the cardinalities $|P_i|$, are
$K$-comparable.

Moreover, by Proposition~\ref{proposition:comparable-calculus},
each set $P_i$ has doubling constant
\[
    L_i \ll K^{O(1)}.
\]

\medskip

The relative density $\eta$ of $A_0$ inside the algebraic intersection of
the sets $P_i$ satisfies
\[
    \eta
    :=
    \frac{|A_0|}
    {|\sqrt[r_1]{P_1} \cap \cdots \cap \sqrt[r_k]{P_k}|}
    \geqslant (1 - o(1)) \frac{k^k}{K^k}.
\]

\medskip

Let $A_1 := A \setminus A_0$.
Fix an index $i$, and choose an index $j = j(i)$ with $j \neq i$.
From the bound
\[
    |A^{r_1} + \cdots + A^{r_k}| \leqslant K \alpha^{1/k} p
\]
we deduce
\[
    |A_1^{r_i} + A_0^{r_j}|
    \leqslant K_i |P_j|,
    \qquad
    K_i := \frac{K \alpha^{1/k} p}{|P_j|}
    \sim \frac{K \alpha^{1/k}}{\omega_j}
    \ll K^{O(1)}.
\]

\medskip

Since the tuple $(r_1, \ldots, r_k)$ is coprime and bounded,
so is the tuple
\[
    \bigl(r_j/r_1,\ \ldots,\ 1,\ \ldots,\ r_j/r_k\bigr).
\]
Because of this and the inclusion
\[
    A_0^{r_j}
    \subseteq
    P_j \cap \sqrt[r_j/r_1]{P_1} \cap \cdots \cap \sqrt[r_j/r_k]{P_k}
\]
we may apply Lemma~\ref{lemma:X+T-growth-structure}\textup{(b)}
with the following choices:
\[
    Y := A_1^{r_i},
    \quad
    T := A_0^{r_j},
    \quad
    P := P_j,
    \quad
    \eta := \eta,
    \quad
    L := L_j,
    \quad
    K := K_i.
\]
This yields a set $X_i$ such that, after removing
$O(1/(\omega_1 \cdots \omega_k)) \ll O(1/\alpha)$ elements from $A_1$, we have
\[
    A_1^{r_i}
    \subseteq
    X_i + \bigl(P_j -_{\kappa_j^2/2} P_j\bigr),
    \qquad
    |X_i| \ll \frac{K_iL_j^7}{\eta^2} \ll K^{O(1)}.
\]

\medskip

By Proposition~\ref{proposition:F-closed}(3), there exists a set
$D_i \in \mathcal{F}$ such that
\[
    P_j -_{\kappa_j^2/2} P_j \subseteq D_i \subseteq P_j - P_j.
\]
Moreover, by~\eqref{eq:byproduct} we have $|D_i| \gg |P_j|$, so the
cardinalities of the sets $D_i$ are $K$-comparable.
By Proposition~\ref{proposition:F-closed}(2), the set
$X_i + D_i$ belongs to $\mathcal{F}$.
Recall also that $A_0^{r_i} \subseteq P_i$, and define
\[
    P_i' := (X_i + D_i) \cup P_i.
\]
Thus, after removing $O(1/\alpha)$ elements from $A$, we obtain
\[
    A^{r_i} \subseteq P_i',
\]
and the sets $P_i'$ belong to family $\mathcal{F}$ by
Proposition~\ref{proposition:F-closed}(1), and therefore have algebraic intersection property, by Proposition \ref{proposition:aip}.

\medskip

Finally, Proposition~\ref{proposition:comparable-calculus} implies that the
sets $P_i'$ are $K$-compatible, and their cardinalities are
$K$-comparable to $\alpha^{1/k} p$.

\medskip

This completes the proof of Theorem~\ref{thm:100-percent} by taking
$P_i := P_i'$ for $i = 1, \ldots, k$.

\section{Appendix}
\label{sec:appendix_dist}

In this section we discuss further connections between additive distance introduced in Subsection \ref{ssec:new_normI} and other metrics.

First of all, let us remark that 
for any $s>1$ the quantity $\rho(f)$ controls the higher energies \cite{SS_higher}
\begin{equation}\label{f:E_*}
    \E_s (f) := \sum_x (f \circ f)^s (x) \le \rho^{2s-2} (f) \|f\|^2_1 \,,
\end{equation}
and, similarly, if $B\subseteq A$, 
then for the entropies of $A$ and $B$ (see \cite{Ruzsa_entropy,GGMT_Marton}) 
\[
    H(A) := \frac{1}{|A|^{2}} \sum_x (A\circ A) (x) \log \left( \frac{|A|^2}{(A\circ A) (x)} \right)
    =
    2\log |A| - \frac{1}{|A|^{2}} \sum_x (A\circ A) (x) \log (A\circ A) (x) \,,
\]
we have
\begin{equation}\label{f:H(A)_H(B)}
    |H(A) - H(B)| 
    \ll \frac{\sqrt{\rho (A,B)} |B-B|\log |B|}{|B|^2} 
\end{equation}
and thus the entropies  of $A$ and $B$ are close in terms of $\rho(A,B)$. 

\begin{remark}
    Thanks to formula \eqref{f:H(A)_H(B)} we know that entropies of sets with small distance $\rho$ are close. 
    What can be said about the sumsets? 
    It is easy to see that the smallness of the distance $\rho$ does not 
    give us 
    too much. 
    Indeed, let $A \subseteq \Gr$ be a random set, $|A|=\d \Gr$, 
    and 
    $\| (A \circ A) (x) - \d^2 N \cdot \d_0 (x) \|_\infty \le \eps$, where $\eps = C\d^2 N$ and $C>1$ is a large constant.
    Then one can delete $O(\d^{-1})$ points from $A$ such that we still have $\| (A \circ A) (x) - \d^2 N \cdot \d_0 (x) \|_\infty \ll \eps$, 
    and such that $A-A$ misses these $O(\sqrt{N} \eps^{-1/2})$ points. 
\end{remark}

Secondly,  if $\|f\|_\infty \le 1$, then 
using formula \eqref{f:Fourier_*},  it is easy to see that $\rho(f) \le \sqrt{N}$ and hence the condition $\rho(f) = o(\sqrt{N})$ guaranties that all Fourier coefficients of the function $f$ can be estimated nontrivially. 
Further, if the support of $f$ is small, then 
it is easy to construct a function with small quantity $\rho$ but large Fourier coefficients, take, e.g., $f(x)=A(x) - |A|/|H| \cdot H(x)$, where $A\subseteq H$ is a random set and $H\le \F_2^n$, say.
Thus, estimate  \eqref{f:Fourier_*} is useful only for functions with large support.

In another direction, take any random set $A \subseteq \Gr$, $|A| \gg N$ and let $f=f_A (x)+c \d_0 (x)$, $c \sim \sqrt{|A|}$. Then the Fourier transform of $f$ is small, namely, $\| \FF{f} \|_\infty \ll |A|^{1/2+o(1)}$ but $\rho^2 (f) \gg c^2 \gg |A|$ is close to its maximal value. 
Thus there are functions with small Fourier coefficients but large value $\rho (f)$. 
Consider another example. 


\begin{exm}
    Let $A\subseteq \F_p$ be a perfect difference set (see, e.g., \cite{Hall_CT}). 
    Then $\| A\circ A - E \|_\infty = 0$, where $E=1-\d_0 (0)$ but $\| A * A - E \|_\infty = 1$. 
    Thus, it is important to take the convolution $\circ$ in \eqref{def:rho} instead of $*$.
    Similarly, if $B=-A$, then $\| A\circ A - B \circ B\|_\infty = 0$ but $\rho^2 (A-B) = \max_{x} (2 - (A*A)(x) - (A*A) (-x)) \geqslant 1$. 
    So, the functions $A\circ A$, $B\circ B$ can be close in the sense of $L_\infty$, but at the same time quite far from each other in the sense of the pseudometric  $\rho$. 
\end{exm}

Finally, as we have seen above (consult formula \eqref{f:Fourier_*}) if $\rho(f)$ is small, then $f$ has small Fourier coefficients. 
In the usual way  \cite{Gowers_m}, one can say that a set $A\subseteq \Gr$ is $\eps$--uniform in the sense of $\rho$ if $\rho(f_A) = \rho (f,\d) \le \eps$.
It is easy to see that $\rho$--uniformity does not imply the higher--order  $U^k$--uniformity, see \cite{Gowers_m}. Indeed, consider the Gowers example \cite{Gowers_m}, namely, let $A(x) = P(x^2)$, where $P\subset \F_p$ is an arithmetic progression, $|P| \gg p$. 
Then it is known (see \cite{Gowers_m}) that $A(x)$ is not $U^3$--uniform. 
On the other hand, using calculations of Lemma \ref{lemma:convolutions}, one can show that $\rho (f_A) \ll p^{1/4+o(1)}$ and hence $A$ is 
$\rho$--uniform.
Nevertheless, Lemma \ref{lemma:convolutions} provides us examples of sets that are $U^k$--uniform for all $k$, e.g., 
the set $A(x) := P(x) P(x^{-1})$ is $\rho$--uniform inside $P$ and, moreover, {\it all} set $A_{s_1,\dots,s_l} := A \cap (A+s_1) \cap \dots \cap (A+s_l)$ are $\rho$--uniform inside $P$. 


\bibliographystyle{abbrv}

\bibliography{bibliography}{}

\bp

\author{A.S. Semchankau\\
Department of Mathematical Sciences, Carnegie Mellon University\\
Wean Hall 6113, Pittsburgh, PA 15213, USA\\ 
\tt{asemchan@andrew.cmu.edu}
}

\bp 

\author{I.D. Shkredov\\
Department of Mathematics, Purdue University\\ 150 N. University Street, West Lafayette, IN 47907--2067, USA\\ 
\tt{ishkredo@purdue.edu}
}

\end{document}